\documentclass[12pt,notitlepage]{amsart}
\usepackage{amscd}

\advance\voffset by 0.1 truein

\renewcommand{\L}{\mathcal{L}}
\newcommand{\I}{\mathcal{I}}
\newcommand{\F}{\mathcal{F}}
\newcommand{\J}{\mathcal{J}}
\newcommand{\V}{\mathcal{V}}
\renewcommand{\H}{\mathcal{H}}
\newcommand{\G}{\mathcal{G}}
\newcommand{\bigomega}{\hbox{\large $\omega$}}
\renewcommand{\O}{\mathcal{O}}
\newcommand{\M}{\mathcal{M}}
\newcommand{\N}{\mathcal{N}}
\newcommand{\ox}{\otimes}
\renewcommand{\:}{\colon}
\newcommand{\w}{\bigomega}
\newcommand{\lra}{\longrightarrow}
\newcommand{\mult}{\text{\rm mult}}
\newcommand{\risom}{\stackrel{\sim}{\to}}

\def\cocoa{{\hbox{\it C\kern-.13em o\kern-.07em C\kern-.13em o\kern-.15em A}}}

\newtheorem{theorem}{Theorem}[section]
\newtheorem{lemma}[theorem]{Lemma}
\newtheorem{proposition}[theorem]{Proposition}

\theoremstyle{definition}
\newtheorem{remark}[theorem]{Remark}
\newtheorem{example}[theorem]{Example}
\newtheorem{subsct}[theorem]{}
\theoremstyle{plain}

\font\smallrm=cmr8

\begin{document}

\author[{\smallrm Eduardo Esteves and Patr\'\i cia Nogueira}]
{Eduardo Esteves and Patr\'\i cia Nogueira}

\thanks{First author supported by CNPq, Proc.~303797/2007-0 and 
473032/2008-2, and FAPERJ, Proc.~E-26.102.769/2008.}

\title[{\smallrm Generalized linear systems and their Weierstrass points}]
{Generalized linear systems on curves\\ and their Weierstrass points}

\begin{abstract} 
Let $C$ be a projective Gorenstein curve over an 
algebraically closed field of characteristic 0. A generalized linear system on 
$C$ is a pair $(\I,\epsilon)$ consisting of a torsion-free, rank-1 sheaf $\I$ on 
$C$ and a map of vector spaces $\epsilon\:V\to\Gamma(C,\I)$. 
If the system is nondegenerate on every irreducible
component of $C$, we associate to it a 0-cycle $W$, its \emph{Weierstrass cycle}. 
Then we show that for each one-parameter family of curves $C_t$ degenerating to 
$C$, and each family of linear systems $(\L_t,\epsilon_t)$ 
along $C_t$, with $\L_t$ invertible, degenerating to $(\I,\epsilon)$, 
the corresponding Weierstrass 
divisors degenerate to a subscheme whose associated 0-cycle is $W$. We show that 
the limit subscheme contains always an ``intrinsic'' subscheme, canonically 
associated to $(\I,\epsilon)$, but the limit itself depends on the family 
$\L_t$.
\end{abstract}

\maketitle

\section{Introduction}\label{intro}

The study of linear systems on complete, 
smooth curves is the study of the projective 
geometry of those curves, and thus has a long history. 
Notably, Severi attempted to prove what is known today as the 
Brill--Noether theorem in his book \cite{Se}, Anhang G, Sect.~8, pp.~380--390; 
see \cite{HM}, Chapter 5 for an account.

The Brill--Noether theorem is a statement about linear 
systems on general smooth curves. 
Roughly speaking, Severi hoped to prove it by degenerating smooth curves to general 
irreducible, rational, nodal curves, keeping 
track of what happens to linear systems under 
such degeneration. 

Severi did not succeed, but his ideas and efforts led to important 
developments in the second half of last century. 
There are two difficulties with Severi's argument. 
First, line bundles do not necessarily degenerate 
to line bundles. But Kleiman \cite{Kl} 
observed that they do degenerate to torsion-free, 
rank-1 sheaves, which he used to overcome 
this first problem, and give a proof of the theorem for linear systems of 
rank 1. The second problem became a conjecture about secant lines of a 
rational normal curve in \cite{Kl}, which, according 
to \cite{HM}, p.~243, remains as a whole undecided. 
But Griffiths and Harris \cite{GH} bypassed the conjecture by degenerating 
even further, thereby completing the proof of the theorem. 

Later on, Eisenbud and Harris \cite{EH1} 
simplified the argument by considering degenerations to 
irreducible, rational, cuspidal curves. Rather than restricting themselves to 
irreducible curves, they observed that, by means of the semistable reduction of 
these degenerations, the irreducible, rational curves could be replaced by special 
curves of compact type, called flag curves. Studying degenerations of 
linear systems, called limit linear series, to curves of compact type, 
Eisenbud and Harris were able to 
discover many other interesting results; see \cite{EH2}.

Central to the study of linear systems is the study of their Weierstrass 
(or ramification) points. 
These are points of the curve at which at least one divisor (or nonzero section) 
of the linear system vanishes with multiplicity higher than the rank of the system. 
In Eisenbud's and Harris's 
theory of limit linear series, degenerations of Weierstrass points play a 
significant role. The so-called refined limit linear series are essentially those 
for which singular points are not limits of Weierstrass points; see \cite{EH2}, 
Prop.~2.5, p.~350.

In contrast, singular points of irreducible curves, at least those with 
Gorenstein singularities, are always limits of Weierstrass points. Major work 
was done by Widland, Lax and Garcia (\cite{GL}, \cite{Lax1}, \cite{Lax2}, 
\cite{LW}, \cite{W}) in the 1980's and early 1990's to define and study Weierstrass 
points of linear systems on such curves. In particular, Lax showed that these 
Weierstrass points are the limits of the Weierstrass points of linear systems 
on smooth curves degenerating to the given linear system on the singular curve; 
see \cite{Lax2}, Prop.~2, p.~9.

However, as observed earlier, linear systems do not necessarily degenerate to 
linear systems, but rather to vector spaces of sections of torsion-free, rank-1 
sheaves, which we call generalized linear systems; see Subsection \ref{tfrk1}. 
Essentially, generalized linear systems appeared before. For instance, they 
appeared in the theory of ``$r$-special'' subschemes developed in \cite{Kl} and 
in the theory of ``generalized divisors'' developed in \cite{Hart2}. However, to 
our knowledge, Weierstrass points for these systems have never been defined nor 
studied.
 
We fill this gap in the literature as follows. 
Let $C$ be a projective, possibly reducible, 
Gorenstein curve over an algebraically closed 
field of characteristic zero. Let $\I$ be a torsion-free, rank-1 sheaf on $C$ 
and $\epsilon\: V\to\Gamma(C,I)$ a nonzero, injective map of vector spaces. 
We say that $(\I,\epsilon)$ is a nondegenerate generalized linear system. 

Assume that $(\I,\epsilon)$ is strongly nondegenerate, that is, that 
$\epsilon(v)$ is generically nonzero on every irreducible component of $C$ 
for every nonzero $v\in V$; see Subsection \ref{tfrk1}. 
We associate to $(\I,\epsilon)$ a subscheme $Z(\I,\epsilon)$ and a 0-cycle 
$R(\I,\epsilon)$  of $C$; see Subsections \ref{singramcyc} and 
\ref{jetsintr}. We call the first the intrinsic 
Weierstrass scheme and the second the Weierstrass cycle of $(\I,\epsilon)$. 

The 0-cycle $R(\I,\epsilon)$ can be computed by adding certain contributions at the 
singular points of $C$ to the 0-cycle associated to the Weierstrass divisor 
of the linear system induced by $(\I,\epsilon)$ on the normalization of $C$. This 
is a consequence of our Theorem \ref{rambir}.

Our main result is that, if $(\I,\epsilon)$ is a limit of ``true'' linear 
systems $(\L_t,\epsilon_t)$ along a family of curves $C_t$ 
degenerating to $C$, then the Weierstrass 
divisors of the linear systems $(\L_t,\epsilon_t)$, parameterizing weighted 
Weierstrass points, converge to a subscheme of $C$ containing 
$Z(\I,\epsilon)$ and whose associated 0-cycle is $R(\I,\epsilon)$; see 
Theorem~\ref{thm}.

So, the 0-cycle of the limit subscheme is intrinsic to $(\I,\epsilon)$. But the 
limit subscheme itself may depend on the degeneration; see Example~\ref{example}. 
As a matter of fact, we observe that the limit depends only 
on the family of invertible sheaves $\L_t$ degenerating to $\I$; 
see Remark~\ref{rmk}. The information one needs to retain from the degeneration 
is the map $\I^{\ox r+1}\to\J$ to a torsion-free, rank-1 sheaf $\J$ obtained, in 
a sense, as the limit of the identity maps of the $\L^{\ox r+1}_t$. 

This indicates that, as far as Weierstrass points are concerned, 
instead of considering moduli spaces for torsion-free, 
rank-1 sheaves --- the compactified Jacobians of \cite{AK2}, \cite{AK}, \cite{DS}, 
\cite{E4}, \cite{OS} or \cite{Ses}, for instance --- 
it might be necessary to consider moduli spaces 
of sheaves $\I$ with additional structures, so that maps like 
$\I^{\ox r+1}\to\J$ are encoded. 
Pacini and the first author consider 
spaces of the type for nodal curves in \cite{EP}. 

The condition that $(\I,\epsilon)$ be strongly nondegenerate, instead of simply 
nondegenerate, is automatic if $C$ is irreducible, but a strong condition 
otherwise. Indeed, it is rare that linear systems of interest, as the 
canonical systems, degenerate to strongly (generalized) linear systems on 
reducible curves; see \cite{Cat} and \cite{EH2} for discussions on this. 
Different approaches are thus necessary, as those taken in \cite{EE} and 
\cite{EH2} or, more recently, in \cite{Oss}.

We do not know whether Theorem \ref{thm}, our main result, 
can be adapted to hold in positive 
characteristic. In fact, our current knowledge of how Weierstrass points vary in 
families, even of smooth curves, is very limited. Nevertheless, as usual in the 
theory, all of the results in the present paper hold if the characteristic of 
the base field is large enough, for instance, larger than the rank of the 
linear systems considered. 

Throughout the paper we adopt a local approach, defining the relevant sections 
of sheaves, the Wronskians, by a patching construction. A global approach is 
possible, in the spirit found in \cite{Lak} and \cite{LT0}, by using the 
substitutes for the sheaves of principals parts given in \cite{GG}, 
\cite{E1}, \cite{E2}, 
\cite{LT1} or \cite{LT2}. For this, we refer the reader to \cite{Pat}.

Here is an outline of the paper. In Section \ref{BNintro} 
we recall linear systems and 
Weierstrass points on smooth curves. In Section \ref{SingC} 
we recall fundamental classes 
and torsion-free, rank-1 sheaves on singular curves. In Section \ref{BOintro} 
we define the Weierstrass cycle of a generalized linear system. In 
Section \ref{compbir} we 
compare Weierstrass cycles using birational maps. In 
Section \ref{invramsch} we define 
the intrinsic Weierstrass scheme of a generalized linear system. In 
Section~\ref{degtorfree} we 
study families of torsion-free, rank-1 sheaves. Finally, in 
Section~\ref{deglinsys} we prove our 
main result, Theorem \ref{thm}, which gives information on limits of Weierstrass 
divisors.

The present paper is heavily based on the second author's doctor 
thesis \cite{Pat}. We thank Steven Kleiman for many comments and references.

\section{Linear systems on smooth curves}\label{BNintro}

\begin{subsct}\label{notation}\setcounter{equation}{0}
{\it Terminology.} A \emph{curve} is a projective, 
reduced scheme of pure dimension 1 over an algebraically closed field. The 
\emph{arithmetic genus} of a curve $C$ is $h^1(C,\O_C)$. 
A \emph{divisor} is a Cartier divisor. A 
\emph{cycle} is a 0-cycle. A \emph{point} is a closed point, unless specified 
otherwise. 

Given a Cartier divisor $D$ of a curve $C$ over an algebraically closed field $k$, 
and $P\in C$, let $\mult_P(D)$ denote the \emph{multiplicity} of $D$ at $P$; if 
$D$ is given at $P$ by $a/b$, for $a,b\in\O_{C,P}$, then
$$
\mult_P(D)=\dim_k\frac{\O_{C,P}}{(a)}-\dim_k\frac{\O_{C,P}}{(b)},
$$
where $k$ is the base field. Let $[D]$ denote the associated cycle, namely,
$$
[D]:=\sum_{P\in C}\mult_P(D)[P].
$$

Likewise, given a coherent sheaf $\F$ on $C$ with finite support, let 
$\mult_P(\F):=\dim_k\F_P$ for each $P\in C$ and set
$$
[\F]:=\sum_{P\in C}\mult_P(\F)[P].
$$
If $Y\subset C$ is a finite subscheme, set $\mult_P(Y):=\mult_P(\O_Y)$ for 
every $P\in C$ and $[Y]:=[\O_Y]$.

Finally, for a coherent sheaf on $C$, its \emph{torsion} subsheaf is the 
maximum coherent subsheaf with finite support. 
\end{subsct}

\begin{subsct}\label{ramcyc}\setcounter{equation}{0}
{\it Linear systems and Weierstrass points.} 
Let $C$ be a curve of arithmetic genus $g$ over an 
algebraically closed field $k$ of characteristic zero. 
Let $\L$ be an invertible sheaf on $C$ and 
$\epsilon\:V\to\Gamma(C,\L)$ a map of 
vector spaces over $k$. Set $d:=\deg\L$ and $r:=\dim_k V-1$. 

We say that 
$(\L,\epsilon)$ is a 
\emph{linear system} of \emph{degree} $d$ and \emph{rank} $r$. 
We say that $(\L,\epsilon)$ is \emph{nondegenerate} 
if $r\geq 0$ and $\epsilon$ is injective. 
If, moreover, for every irreducible component 
$Y\subseteq C$ the composition
$$
V\lra\Gamma(C,\L)\lra\Gamma(Y,\L|_Y)
$$
of $\epsilon$ with the restriction map is injective, 
then we say that $(\L,\epsilon)$ is \emph{strongly nondegenerate}.

Assume $(\L,\epsilon)$ is strongly nondegenerate. Let $P$ be a 
simple point of $C$, that is, a point 
on the nonsingular locus of $C$. We say that an integer $e$ is an 
\emph{order} of $(\L,\epsilon)$ at $P$ 
if there is a nonzero $v\in V$ such that $\epsilon(v)$ vanishes at $P$ with 
order $e$. If two sections of $\L$ have the same 
order at $P$, a certain $k$-linear combination of them will be zero 
or have higher order. Thus there are exactly $r+1$ 
orders of $(\L,\epsilon)$ at $P$. Putting them in increasing order we get a sequence
$$
e_0(P),e_1(P)\dots,e_r(P),
$$
called the {\it order sequence} of $(\L,\epsilon)$ at $P$. 

For each simple $P\in C$, put
$$
e(P):=\sum_{i=0}^r\big(e_i(P)-i\big).
$$
We call $P$ a \emph{Weierstrass point} of $(\L,\epsilon)$ if $e(P)>0$. 

If $C$ is nonsingular, we call the cycle
$$
R(\L,\epsilon):=\sum_{P\in C}e(P)[P]
$$
the \emph{Weierstrass cycle} of $(\L,\epsilon)$. 
That it is indeed a cycle, that is, 
that there are only finitely many Weierstrass points of $(\L,\epsilon)$, will be 
seen in Subsection \ref{cycdiv}.
\end{subsct}

\begin{subsct}\label{ramdiv}\setcounter{equation}{0}
{\it The Weierstrass divisor.} Keep the setup of 
Subsection \ref{ramcyc}. In particular, assume that $(\L,\epsilon)$ is a 
strongly nondegenerate linear system.

The sheaf of K\"ahler differentials $\Omega^1_C$ is invertible on the 
nonsingular locus of $C$. Thus the nonsingular locus can be covered by open 
subschemes $U$ for which 
$\Omega^1_U$ and $\L|_U$ are trivial. For such a $U$, 
let $\mu\in\Gamma(U,\Omega^1_C)$ and 
$\sigma\in\Gamma(U,\L)$ be sections generating 
$\Omega^1_U$ and $\L|_U$. 

Fix a basis $\beta=(v_0,\dots,v_r)$ of $V$. Then there are 
regular functions $f_0,\dots,f_r$ on $U$ such that 
$\epsilon(v_i)|_U=f_i\sigma$ for each $i=0,\dots,r$. Let $\partial$ be the 
$k$-linear 
derivation of $\Gamma(U,\O_C)$ such that 
$dh=\partial h\mu$ for each regular function $h$ on $U$. 
Form the Wronskian determinant:
$$
w(\beta,\sigma,\mu):=\begin{vmatrix}
f_0 & \dots & f_r\\
\partial f_0 & \dots & \partial f_r\\
\vdots & \ddots &\vdots\\
\partial^rf_0 & \dots & \partial^rf_r\\
\end{vmatrix}.
$$

If $\sigma'$ and $\mu'$ are other generators of $\L|_U$ and 
$\Omega^1_U$, respectively, then 
$\sigma'=a\sigma$ and $\mu'=b\mu$ for certain everywhere 
nonzero regular functions $a$ and $b$ on $U$. Also, if 
$\beta'=(v'_0,\dots,v'_r)$ is another basis of $V$, then $\beta'=\beta M$, 
where $M$ is an invertible matrix of size $r+1$ and entries in $k$. By 
the multilinearity of the determinant and the product rule of derivations,
$$
w(\beta',\sigma',\mu'):=\begin{vmatrix}
af'_0 & \dots & af'_r\\
b\partial(af'_0) & \dots & b\partial(af'_r)\\
\vdots & \ddots &\vdots\\
(b\partial)^r(af'_0) & \dots & (b\partial)^r(af'_r)\\
\end{vmatrix}=ca^{r+1}b^{\binom{r+1}{2}}w(\beta,\sigma,\mu),
$$
where $c:=\det M$.

Thus the $w(\beta,\sigma,\mu)$ patch up to a section $w$ of 
$$
\L^{\ox r+1}\ox(\Omega^1_C)^{\ox\binom{r+1}{2}}
$$
over the nonsingular locus of $C$, well-defined up to multiplication by an element 
of $k^*$.

Assume  $C$ is nonsingular. Then we call $w$ a {\it Wronskian} of $(\L,\epsilon)$. 
The zero scheme of $w$ is denoted by 
$W(\L,\epsilon)$ and called the 
{\it Weierstrass divisor} of $(\L,\epsilon)$. Though $w$ is only defined 
modulo $k^*$, the divisor $W(\L,\epsilon)$ is well-defined. 

That $W(\L,\epsilon)$ is indeed a divisor will be 
seen in Subsection \ref{cycdiv}. Since $\L$ has degree $d$ and 
$\Omega^1_C$ has degree $2g-2n$, 
where $n$ is the number of connected components of $C$, 
it follows that
\begin{equation}\label{pluck}
\deg[W(\L,\epsilon)]=\big(r+1\big)\big(d+r(g-n)\big),
\end{equation}
a formula known as the \emph{Pl\"ucker formula}.
\end{subsct}

\begin{subsct}\label{cycdiv}\setcounter{equation}{0}
{\it From divisor to cycle.} Keep the setup of 
Subsection \ref{ramcyc}. 

The relation between the Weierstrass cycle and divisor is simple: the 
cycle is that associated to the divisor. 
To prove this, let $P$ be a simple point of $C$. Let $t$ be a local parameter of 
$C$ at $P$. Then $t$ is a 
regular function on an open neighborhood 
$U\subseteq C$ of $P$. 
Shrinking $U$ around $P$ if necessary, we may assume that 
$dt$ generates $\Omega^1_U$. Also, we may assume there 
is $\sigma\in\Gamma(U,\L)$ generating $\L|_U$.

There are $v_0,\dots,v_r\in V$ such that $\epsilon(v_i)$ vanishes at $P$ with 
order $e_i(P)$ for $i=0,1,\dots,r$. Shrinking 
$U$ around $P$ if necessary, we may assume that there are 
everywhere nonzero regular functions $u_0,\dots,u_r$ on 
$U$ such that 
$$
\epsilon(v_i)|_U=u_it^{e_i(P)}\sigma
$$
for each $i=0,\dots,r$. Since the orders $e_i(P)$ are distinct, it 
follows that $\beta:=(v_0,\dots,v_r)$ is a basis 
of $V$.

The Wronskian determinant 
$w(\beta,\sigma,dt)$ has the form:
$$
w(\beta,\sigma,dt)=\begin{vmatrix}
u_0t^{e_0(P)} & \dots & u_rt^{e_r(P)}\\
\frac{d}{dt}(u_0t^{e_0(P)}) & \dots & 
\frac{d}{dt}(u_rt^{e_r(P)})\\
\vdots & \ddots &\vdots\\
\frac{d^r}{dt^r}(u_0t^{e_0(P)}) & \dots & 
\frac{d^r}{dt^r}(u_rt^{e_r(P)})\\
\end{vmatrix}.
$$
Using the multilinearity of the determinant, the 
product rule of derivations, and the fact that 
$\frac{d}{dt}(t^j)=jt^{j-1}$ 
for each integer $j\geq 1$, we 
get 
$$
w(\beta,\sigma,dt)=t^{e(P)}v,
$$
where $e(P)=\sum_i(e_i(P)-i)$, and where $v$ is a regular function on $U$ 
such that 
$$
v(P)=\begin{vmatrix}
1 & 1 & \dots & 1\\
e_0(P) & e_1(P) & \dots & e_r(P)\\
\vdots & \vdots & \ddots & \vdots\\
e_0(P)^r & e_1(P)^r & \dots & e_r(P)^r
\end{vmatrix}
\prod_{i=0}^ru_i(P).
$$
Since the $e_i(P)$ are 
distinct, the Van der Monde determinant is nonzero. So $v(P)\neq 0$, 
and hence $w(\beta,\sigma,dt)$ vanishes at $P$ with order $e(P)$. 

Since the above reasoning is valid for every simple point $P$ of $C$, it 
follows that the section $w$ of Subsection \ref{ramdiv} has finitely many zeros. 
In particular, if $C$ is nonsingular, then $W(\L,\epsilon)$ is indeed a divisor. 
Furthermore, since for each $P\in C$ the multiplicity of $W(\L,\epsilon)$ at $P$ is 
$e(P)$, the cycle associated to $W(\L,\epsilon)$ is $R(\L,\epsilon)$.
\end{subsct} 

\section{Singular curves}\label{SingC}

\begin{subsct}\label{canmap}\setcounter{equation}{0}
{\it The fundamental class.} Let $C$ be a curve over an algebraically closed 
field $k$. Let $\M$ be the sheaf of meromorphic differentials of $C$. 
Given a point $P\in C$, we have
$$
\M_P=\prod_{i=1}^m\Omega^1_{K_i/k},
$$
where $K_1,\dots,K_m$ are the fields of functions of the irreducible components 
of $C$ containing $P$. 

Let $\w_C$ denote Rosenlicht's sheaf of regular differentials of $C$; see 
\cite{Conrad}, Section 5.2, p.~226. It is the subsheaf 
of $\M$ satisfying the following 
property: A meromorphic differential $\tau$ is in 
$\w_{C,P}$ for $P\in C$ if 
$$
\sum_{i=1}^m\text{Res}_{Q_i}(f\tau)=0
$$
for every $f\in\O_{C,P}$, where $Q_1,\dots,Q_m$ are the points on the normalization 
of $C$ mapping to $P$. From its defining property, $\w_C$ contains 
those meromorphic differentials arising from K\"ahler differentials of $C$. Then 
there is a natural map $\gamma\:\Omega^1_C\to\w_C$, 
called the \emph{fundamental class}. (See 
\cite{Li}, p.~39, for a justification of the name.) From 
the defining property of $\w_C$, this map is an isomorphism on the 
nonsingular locus of $C$. 

It follows from \cite{Conrad}, Thm.~5.2.3, p.~230, that $\w_C$ is a dualizing 
sheaf for $C$. Thus, if $C$ is Gorenstein, that is, if the local rings 
$\O_{C,P}$ are Gorenstein for all $P\in C$, then it follows from 
\cite{Hart}, Prop.~9.3, p.~296, that $\w_C$ is invertible. (Rosenlicht's own 
proofs of 
these two results are \cite{Rosen}, Thm.~8, p.~177 and Thm.~10, p.~179, 
respectively, for irreducible curves; see comments on Section~4 of loc.~cit., 
especially on p.~185, for the case of reducible curves.)
\end{subsct}

\begin{subsct}\label{tfrk1sh}\setcounter{equation}{0}
{\it Torsion-free, rank-1 sheaves.} Let $C$ be a curve of arithmetic genus 
$g$. Let $\I$ be a coherent sheaf on $C$. We say that $\I$ is 
\emph{torsion-free} if the 
generic points of $C$ are its only associated points. Also, $\I$ is said to be 
\emph{rank-1} if $\I$ is invertible on a dense open subset of $C$. Putting it in 
different words, $\I$ is torsion-free, rank-1 if it is isomorphic 
to a sheaf of fractional ideals, that is, a coherent subsheaf of the 
sheaf of rational functions which is everywhere nonzero. 

Invertible sheaves are torsion-free, rank-1. Conversely, torsion-free, 
rank-1 sheaves are invertible on the nonsingular locus of $C$.

Assume $\I$ is torsion-free, rank-1. 
The \emph{degree} of $\I$ is denoted by $\deg\I$ and defined by 
$$
\deg\I:=\chi(\I)-\chi(\O_C)=h^0(C,\I)-h^1(C,\I)+g-n,
$$
where $n$ is the number of connected components of $C$. 
If $\I$ is invertible, then $\deg\I$ is the usual degree, by Riemann--Roch. 

If $\L$ is an invertible sheaf, then
\begin{equation}\label{degdeg}
\deg\I\ox\L=\deg\I+\deg\L.
\end{equation}
Indeed, $\I\ox\L$ is torsion-free, rank-1. Also, if $P$ is a simple point 
of $C$, 
$$
\deg\I\ox\O_C(P)=\chi(\I\ox\O_C(P))-\chi(\O_C)=\chi(\I)+1-\chi(\O_C)=\deg\I+1.
$$
Since $\L\cong\O_C(\sum_i\pm P_i)$, where the $P_i$ are simple points of $C$, 
Equation~\eqref{degdeg} follows from applying the last equation repeatedly.

From its definition, $\w_C$ is torsion-free, rank-1. 
Furthermore, since $\I$ has depth 1 at every point of $C$, it follows 
from \cite{AK}, (6.5.3), p.~96, that
\begin{equation}\label{Ext10}
Ext^1(\I,\w_C)=0.
\end{equation}
Thus the spectral sequence 
associated to the composition of 
functors $Hom(-,\w_C)$ and $\Gamma(C,-)$ degenerates to yield 
$$
\text{Ext}^1(\I,\w_C)=H^1(C,Hom(\I,\w_C)).
$$
It follows that $Hom(\I,\w_C)$ satisfies ``duality properties'' 
with respect to $\I$. More 
precisely,
$$
h^0(C,Hom(\I,\w_C))=\dim_k\text{Hom}(\I,\w_C)=h^1(C,\I)
$$
and
$$
h^1(C,Hom(\I,\w_C))=\dim_k\text{Ext}^1(\I,\w_C)=h^0(C,\I).
$$
In particular, 
\begin{equation}\label{chiI}
\chi(Hom(\I,\w_C))=-\chi(\I).
\end{equation}

Equation \eqref{chiI} implies that the natural map
$$
\I\lra Hom(Hom(\I,\w_C),\w_C)
$$
is an isomorphism. Indeed, it is an isomorphism on 
the nonsingular locus of $C$, where $\I$ is invertible, whence injective with 
cokernel supported on a finite set. And the cokernel is zero 
because its Euler characteristic is zero, since
$$
\chi(Hom(Hom(\I,\w_C),\w_C))=-\chi(Hom(\I,\w_C))=\chi(\I).
$$

It follows from \eqref{chiI} as well that
$$
\deg\w_C=\chi(\w_C)-\chi(\O_C)=-2\chi(\O_C)=2g-2n.
$$
Also, if $C$ is Gorenstein, then $\w_C$ is invertible, and thus
\begin{equation}\label{deg-deg}
\begin{aligned}
\deg Hom(\I,\O_C)=&\deg Hom(\I,\w_C)-\deg\w_C\\
=&\chi(Hom(\I,\w_C))-\chi(\w_C)\\
=&-\chi(\I)+\chi(\O_C)\\
=&-\deg\I,
\end{aligned}
\end{equation}
where the first equality follows from \eqref{degdeg} and the natural isomorphism
$Hom(\I,\w_C)\cong Hom(\I,\O_C)\ox\w_C$.
\end{subsct}

\section{Linear systems on singular curves}\label{BOintro}

\begin{subsct}\label{ramifsing}\setcounter{equation}{0}
{\it Weierstrass divisors.} Keep the setup of 
Subsection \ref{ramcyc}. In particular, assume that $(\L,\epsilon)$ is a 
strongly nondegenerate linear system. 

If $C$ is singular then $\Omega^1_C$ is not invertible, and hence the 
reasoning in Subsections \ref{ramdiv} and \ref{cycdiv} 
cannot be directly applied everywhere.

However, assume $C$ is Gorenstein, and let $\gamma\:\Omega^1_C\to\w_C$ be the 
fundamental class. The curve $C$ can be covered by open subschemes $U$ for which 
$\w_C|_U$ is trivial. For such a $U$, let $\mu\in\Gamma(U,\w_C)$ be a 
section generating $\w_C|_U$. Then there is a $k$-linear derivation $\partial$ 
of $\Gamma(U,\O_C)$ such that 
$\gamma df=\partial f\mu$ for each regular funtion $f$ on $U$.

We can now reason exactly as in Subsection \ref{ramdiv}, to obtain the 
zero scheme of a global section $w$ of 
$$ 
\L^{\ox r+1}\ox\w_C^{\ox\binom{r+1}{2}}.
$$
Denote this zero scheme by $W(\L,\epsilon)$. 
Since $\gamma$ is an isomorphism on the nonsingular locus of $C$, the 
reasoning in Subsection \ref{cycdiv} can be applied. In 
particular, it follows that $W(\L,\epsilon)$ is a divisor. And the multiplicity of 
$W(\L,\epsilon)$ at a simple $P\in C$ is $e(P)$, as defined 
in Subsection \ref{ramcyc}.

We call $w$ a \emph{Wronskian} and 
$W(\L,\epsilon)$ the \emph{Weierstrass divisor} of $(\L,\epsilon)$. 
Since $\w_C$ has degree $2g-2n$, where $n$ is the number of connected components 
of $C$, Pl\"ucker formula \eqref{pluck} holds. 
\end{subsct}

\begin{subsct}\label{tfrk1}\setcounter{equation}{0}
{\it Generalized linear systems.} Let $C$ be a curve and $\I$ a torsion-free, 
rank-1 sheaf on $C$. 

There is an injection $\I\hookrightarrow\L$ into an invertible sheaf $\L$. 
Indeed, let $\O_C(1)$ be an ample invertible sheaf on $C$. Then, for 
$m$ sufficiently large, 
$Hom(\I,\O_C)(m)$ is generated by global sections. In particular, it has a 
global section which is nonzero at the (finitely many) generic points of $C$. This 
section corresponds to an injection $\I\hookrightarrow\O_C(m)$. (Notice that, since 
$C$ is reduced and $\I$ and $\L$ are rank-1, any injection 
$\I\hookrightarrow\L$ is generically an isomorphism.)

Let $\epsilon\:V\to\Gamma(C,\I)$ be a map of vector spaces. Set $d\:=\deg\I$ and 
$r:=\dim V-1$. We say that $(\I,\epsilon)$ is a (\emph{generalized}) 
\emph{linear system} of 
\emph{degree} $d$ and \emph{rank} $r$. As in Subsection \ref{ramcyc}, 
we say that $(\I,\epsilon)$ is \emph{nondegenerate} if $r\geq 0$ and $\epsilon$ is 
injective. 

For any subcurve $Y\subseteq C$, that is, any reduced union of irreducible 
components of $C$, let $\I^Y$ denote the restriction $\I|_Y$ modulo torsion. 
In other words, let $\I^Y$ be the image of the natural map
$$
\I\lra\prod_{i=1}^m\I_{\xi_i},
$$
where $\xi_1,\dots,\xi_m$ are the generic points of $Y$. We say that $(\I,\epsilon)$ 
is \emph{strongly nondegenerate} if for each irreducible component $Y\subseteq C$ 
the composition
$$
V\lra\Gamma(C,\I)\to\Gamma(Y,\I^Y)
$$
of $\epsilon$ with the map induced by the quotient map $\I\to\I^Y$ is injective.

Let $\varphi\:\I\hookrightarrow\L$ be an injection into an invertible 
sheaf. Let $\epsilon'$ be the composition
$$
\begin{CD}
V @>\epsilon >> \Gamma(C,\I) @>\Gamma(\varphi)>> \Gamma(C,\L).
\end{CD}
$$
Since $\varphi$ is generically an isomorphism, the induced linear system 
$(\L,\epsilon')$ is (strongly) nondegenerate if and only if $(\I,\epsilon)$ is. 
\end{subsct}

\begin{subsct}\label{singramcyc}\setcounter{equation}{0}
{\it Weierstrass cycles.} Let $C$ be a Gorenstein curve of arithmetic genus $g$ 
over an algebraically closed field of characteristic zero. 
Let $\I$ be a torsion-free, rank-1 sheaf on 
$C$ and $\epsilon\:V\to\Gamma(C,\I)$ a map of vector spaces. 
Assume $(\I,\epsilon)$ is 
a strongly nondegenerate linear system. Let $r$ denote its rank and $d$ its degree.

Let $\varphi\:\I\hookrightarrow\L$ be an injection into an invertible 
sheaf. Let $\epsilon'$ be the composition
$$
\begin{CD}
V @>\epsilon >> \Gamma(C,\I) @>\Gamma(\varphi)>> \Gamma(C,\L).
\end{CD}
$$
Also, let $Y\subseteq C$ be the closed subscheme such that 
$\text{Im}(\varphi)=\I_{Y/C}\L$. Since $\varphi$ is generically an 
isomorphism, $Y$ is finite. 

Define the 
\emph{Weierstrass cycle} of $(\I,\epsilon)$ by
$$
R(\I,\epsilon):=[W(\L,\epsilon')]-(r+1)[Y].
$$
By Proposition \ref{indepI} below, the cycle $R(\I,\epsilon)$ does not 
depend on the choice of injection $\varphi$.

The Pl\"ucker formula holds for $R(\I,\epsilon)$, that is, 
\begin{equation}\label{pluckI}
\deg R(\I,\epsilon)=\big(r+1\big)\big(d+r(g-n)\big),
\end{equation}
where $n$ is the number of connected components of $C$. 
Indeed, from the usual Pl\"ucker formula \eqref{pluck}, we get
\begin{equation}\label{pluckprior}
\deg R(\I,\epsilon)=(r+1\big)\big(d'+r(g-n)\big)-(r+1)\deg[Y],
\end{equation}
where $d':=\deg\L$. However, since $\L$ is invertible, the additiveness of the 
Euler characteristic yields $\chi(\L)=\chi(\I)+\deg[Y]$. Using this equation in 
\eqref{pluckprior} we get \eqref{pluckI}.
\end{subsct}

\begin{lemma}\label{inj2} 
Let $C$ be a curve, $\I$ a torsion-free, rank-1 sheaf on $C$ and 
$\varphi\:\I\hookrightarrow\L$ and $\psi\:\I\hookrightarrow\M$ injections into 
invertible sheaves. Then there are an invertible sheaf $\N$ and injections 
$\lambda\:\L\hookrightarrow\N$ and $\mu\:\M\hookrightarrow\N$ such that 
$\lambda\varphi=\mu\psi$.
\end{lemma}

\begin{proof} Let $U\subseteq C$ be an open dense subscheme such that 
$\I|_U$, $\L|_U$ and $\M|_U$ are trivial, and let $\sigma$, $\tau$ and $\upsilon$, 
respectively, be sections over $U$ 
generating these restrictions. Then there are regular 
functions $a$ and $b$ on $U$ such that $\varphi(\sigma)=a\tau$ and 
$\psi(\sigma)=b\upsilon$. Since $\varphi$ and $\psi$ are injections, $a$ and $b$ 
have finitely many zeros on $U$. 

Define maps $\lambda'\:\L|_U\to\O_U$ and $\mu'\:\M|_U\to\O_U$ by setting 
$\lambda'(\tau)=b$ and $\mu'(\upsilon)=a$. Since $a$ and $b$ have finitely many 
zeros on $U$, the maps $\lambda'$ and $\mu'$ are injective. And clearly 
$\lambda'\varphi|_U=\mu'\psi|_U$. 

Let $\O_C(1)$ be an ample sheaf on $C$. Since $C-U$ is finite, there are an integer 
$\ell$ and a global section $f$ of $\O_C(\ell)$, nonzero on every irreducible 
component 
of $C$, such that the open subscheme 
$$
C_f:=\{P\in C\,|\,f(P)\neq 0\}
$$
is contained in $U$. Since $C_f$ is also dense in $C$, we may assume that 
$U=C_f$; and that $\ell=1$.

View $\lambda'$ as a section of $Hom(\L,\O_C)$ over $U$. There are 
an integer $m$ and a global section $\lambda$ of $Hom(\L,\O_C)(m)$ such that 
$\lambda|_U=\lambda'\ox f^{\ox m}$. We may view $\lambda$ as a map 
$\L\hookrightarrow\O_C(m)$, which is on $U$ the composition of 
$\lambda'$ with the multiplication by $f^{\ox m}$. Then $\lambda$ is an injection 
because $\lambda'$ is and $U$ is dense in $C$. 

Likewise, there are an integer $n$ and an injection $\mu\:\M\hookrightarrow\O_C(n)$ 
whose restriction to $U$ is the composition of $\mu'$ with the multiplication by 
$f^{\ox n}$. Up to replacing $m$ and $n$ by $\max(m,n)$, we may assume that $m=n$. 

Set $\N:=\O_C(m)$. Since $\lambda'\varphi|_U=\mu'\psi|_U$, the compositions 
$\lambda\varphi$ and $\mu\psi$ agree on $U$. Since $U$ is dense in $C$, they agree 
everywhere.
\end{proof}

\begin{proposition}\label{indepI} 
Let $C$ be a Gorenstein curve over an algebraically closed field 
$k$ of characteristic zero. Let $(\I,\epsilon)$ be a 
strongly nondegenerate (generalized) linear system of $C$. 
Then the cycle $R(\I,\epsilon)$ does not 
depend on the choice of injection of $\I$ into an invertible sheaf.
\end{proposition}

\begin{proof} Let $\varphi\:\I\hookrightarrow\L$ and $\psi\:\I\hookrightarrow\M$ 
be injections into invertible sheaves. By Lemma \ref{inj2}, there are an 
invertible sheaf $\N$ and injections 
$\lambda\:\L\hookrightarrow\N$ and $\mu\:\M\hookrightarrow\N$ such that 
$\lambda\varphi=\mu\psi$. Let $(\L,\epsilon')$, $(\M,\epsilon'')$ and 
$(\N,\epsilon''')$ be the induced linear systems.

Let $r$ be the rank of $(\I,\epsilon)$. 
Let $T$, $X$, $Y$ and $Z$ be the closed subschemes of 
$C$ such that
$$
\varphi(\I)=\I_{T/C}\L,\  
\psi(\I)=\I_{X/C}\M,\  
\lambda(\L)=\I_{Y/C}\N\  \text{and}\  
\mu(\M)=\I_{Z/C}\N.
$$
Since $\lambda\varphi=\mu\psi$, we have $\I_{T/C}\I_{Y/C}=\I_{X/C}\I_{Z/C}$. Thus, 
since $Y$ and $Z$ are divisors,
\begin{equation}\label{XYZ}
[T]+[Y]=[X]+[Z].
\end{equation}

Furthermore, given regular functions $f,f_0,\dots,f_r$ on an open subscheme 
$U$ of $C$, and a $k$-linear derivation $\partial$ of $\Gamma(U,\O_C)$, the 
multilinearity of the determinant and the product rule of derivations yield
$$
\begin{vmatrix}
ff_0 & \dots & ff_r\\
\partial(ff_0) & \dots & \partial(ff_r)\\
\vdots & \ddots &\vdots\\
\partial^r(ff_0) & \dots & \partial^r(ff_r)\\
\end{vmatrix}=
f^{r+1}\begin{vmatrix}
f_0 & \dots & f_r\\
\partial(f_0) & \dots & \partial(f_r)\\
\vdots & \ddots &\vdots\\
\partial^r(f_0) & \dots & \partial^r(f_r)\\
\end{vmatrix}.
$$
Thus
\begin{equation}\label{VLMN}
W(\N,\epsilon''')=W(\L,\epsilon')+(r+1)Y=W(\M,\epsilon'')+(r+1)Z.
\end{equation}
Finally, combining \eqref{XYZ} and \eqref{VLMN},
\begin{align*}
[W(\L,\epsilon')]-(r+1)[T]=&[W(\N,\epsilon''')]-(r+1)([Y]+[T])\\
=&[W(\N,\epsilon''')]-(r+1)([X]+[Z])\\
=&[W(\M,\epsilon'')]-(r+1)[X].
\end{align*}
\end{proof}

\section{Comparisons under birational maps}\label{compbir}

\begin{subsct}\label{bir}\setcounter{equation}{0}
{\it Birational maps.} Let $b\:C^\dagger\to C$ be a birational map between 
curves. Let $b^\#\:\O_C\to b_*\O_{C^\dagger}$ denote the comorphism.

Since $b$ is birational, $b^\#$ is injective. So we may view 
$\O_C$ inside $b_*\O_{C^\dagger}$ under $b^\#$. Let 
$$
R_b:=\Big[\frac{b_*\O_{C^\dagger}}{\O_C}\Big].
$$
Also, let $\F$ be the \emph{conductor ideal} of $b$, that is, the annihilator of 
$\text{Coker}(b^\#)$. Since $b$ is birational, $\F$ is 
torsion-free, rank-1. Also, $\F$ is a sheaf of ideals of $\O_C$ as well 
as of $b_*\O_{C^\dagger}$. In other words, $\F=b_*(\F\O_{C^\dagger})$ 
as subsheaves of 
$b_*\O_{C^\dagger}$. Let $Z\subset C$ be the subscheme defined 
by $\F$ and $Z^\dagger:=b^{-1}(Z)$. We call $Z$ the \emph{conductor scheme}. 
Then, since $b$ is finite,
\begin{equation}\label{piZR}
\begin{aligned}
b_*[Z^\dagger]=&b_*\Big[\frac{\O_{C^\dagger}}{\F\O_{C^\dagger}}\Big]=
\Big[\frac{b_*\O_{C^\dagger}}{\F}\Big]\\
=&\Big[\frac{b_*\O_{C^\dagger}}{\O_C}\Big]+
\Big[\frac{\O_C}{\F}\Big]=R_b+[Z].
\end{aligned}
\end{equation}

Let $\M$ and $\M^\dagger$ be the sheaves of meromorphic differentials of 
$C$ and $C^\dagger$, respectively. Since $b$ is birational, 
there is a natural isomorphism 
$b_*\M^\dagger\risom\M$. From the definition of regular 
differentials in Subsection \ref{ramifsing}, 
it follows that this isomorphism carries $b_*\w_{C^\dagger}$ into 
$\w_C$. Let $\iota\: b_*\w_{C^\dagger}\to\w_C$ denote the induced inclusion. 
Then we get the following natural commutative diagram:
\begin{equation}\label{WCC'}
\begin{CD}
\O_C @>d>> \Omega^1_C @>\gamma >> \w_C\\
@Vb^\#VV @V\alpha VV @A\iota AA\\
b_*\O_{C^\dagger} @>b_*d^\dagger >> b_*\Omega^1_{C^\dagger} 
@>b_*\gamma^\dagger >> b_*\w_{C^\dagger}
\end{CD}
\end{equation}
where $d$ and $d^\dagger$ are the universal derivations, 
$\gamma$ and $\gamma^\dagger$ 
are the fundamental classes, and $\alpha$ is the adjoint 
to the natural pullback map of 
differentials. 

Assume now that $C$ is Gorenstein in a neighborhood of $Z$. 
We claim that $R_b=[Z]$, and thus, by \eqref{piZR},
\begin{equation}\label{FZZ}
b_*[Z^\dagger]=2R_b.
\end{equation}
The claim follows from 
local duality. It follows as well from global duality, as we explain now. 
First of all, 
$b=b_1b_2\cdots b_m$, where each $b_i$ is birational, with conductor 
scheme $Z_i$ supported at a single point $P_i$, and such that 
$(b_i\dots b_{j-1})(P_j)\neq P_i$ for each $i$ and $j$ with $i<j$. Since 
it is enough to show that $R_{b_i}=[Z_i]$ for each $i$, we may assume, to prove 
the claim, that $Z$ is supported at a single point. Since $Z$ and $R_b$ have the 
same support, we need only show that $\deg R_b=\deg[Z]$. But
\begin{equation}\label{RpiZZ}
\begin{aligned}
\deg[Z]=&\chi(\O_C)-\chi(\F)\\
=&\chi(\w_C)-\chi(\F\w_C)\\
=&\chi(\w_C)-\chi(Hom(b_*\O_{C^\dagger},\w_C))\\
=&-\chi(\O_C)+\chi(b_*\O_{C^\dagger})\\
=&\deg R_b,
\end{aligned}
\end{equation}
where the second equality holds because $\w_C$ is invertible in a neighborhood of 
$Z$, the third because
$\F=Hom(b_*\O_{C^\dagger},\O_C)$, and the fourth by (global) duality.

Furthermore, we claim that 
\begin{equation}\label{Fw}
\iota(b_*\w_{C^\dagger})=\F\w_C.
\end{equation}
Indeed, $\iota(b_*\w_{C^\dagger})\subseteq\F\w_C$ because $\w_C$ is invertible in a 
neighborhood of $Z$, and hence $\F\w_C$ is the maximum subsheaf of $\w_C$ which is 
a sheaf of $b_*\O_{C^\dagger}$-modules. Furthermore,
$$
\chi(\w_C)-\chi(\F\w_C)=\chi(\O_{C^\dagger})-\chi(\O_C)=
\chi(\w_C)-\chi(b_*\w_{C^\dagger}),
$$
where the first equality follows from \eqref{RpiZZ} and the second by duality and 
the finiteness of $b$. Thus, since the source and target of the inclusion 
$b_*\w_{C^\dagger}\hookrightarrow\F\w_C$ have the same Euler characteristic, and 
the quotient has finite support, \eqref{Fw} follows.
\end{subsct}

\begin{subsct}\label{birI}\setcounter{equation}{0}
{\it Generalized linear systems and birational maps.} Keep the setup of 
Subsection \ref{bir}. 

For each torsion-free, rank-1 sheaf $\I$ on $C$, let $\I^b$ denote 
$b^*\I$ modulo torsion. Then there is a natural injection 
$h^b_\I\:\I\to b_*\I^b$. We let $R_b(\I)$ denote the 
cycle associated to $\text{Coker}(h^b_\I)$.

Notice that $h_{\O_C}^b$ is simply the comorphism of $b$. So 
$R_b=R_b(\O_C)$. Also, if $\I$ is 
invertible then so is $b^*\I$, and hence $\I^b=b^*\I$. Moreover, 
$h^b_\I=h^b_{\O_C}\ox\I$, and thus $R_b(\I)=R_b$. 

Let $\epsilon\:V\to\Gamma(C,\I)$ be a map of vector spaces. Then 
$(\I^b,\epsilon^\dagger)$ is a (generalized) linear system, where 
$\epsilon^\dagger$ is the composition
$$
\begin{CD}
V @>\epsilon >> \Gamma(C,\I) @>\Gamma(h^b_\I)>> \Gamma(C,b_*\I^b) @>=>> 
\Gamma(C^\dagger,\I^b).
\end{CD}
$$
We say that $(\I^b,\epsilon^\dagger)$ is \emph{induced} by $(\I,\epsilon)$. 
Since $C$ is birational, $h^b_\I$ is generically an isomorphism. Thus, 
if $(\I,\epsilon)$ is (strongly) nondegenerate, so is $(\I^b,\epsilon^\dagger)$. 
\end{subsct}

\begin{theorem}\label{rambir} 
Let $b\:C^\dagger\to C$ be a birational map between 
Gorenstein curves over an algebraically closed field of characteristic zero. 
Let $(\I,\epsilon)$ be a (generalized) linear system of rank $r$ 
on $C$, and $(\I^b,\epsilon^\dagger)$ the induced system on $C^\dagger$. Then
$$
R(\I,\epsilon)-b_*R(\I^b,\epsilon^\dagger)=(r+1)^2R_b-(r+1)R_b(\I).
$$ 
\end{theorem}

\begin{proof} Keep the notations of Subsections \ref{bir} and \ref{birI}. 
Assume first that $\I$ is invertible. Set $\L:=\I$. 
Then we need to prove that
\begin{equation}\label{VLL}
R(\L,\epsilon)-b_*R(b^*\L,\epsilon^\dagger)=(r+1)rR_b.
\end{equation}

Since $C$ and $C^\dagger$ are Gorenstein, it follows from 
\eqref{Fw} that $\F\O_{C^\dagger}$ is 
invertible. In other words, $Z^\dagger$ is an effective divisor. 
Furthermore, $C$ can be covered by affine open subschemes $U$ 
for which $\w_C|_U$, $\L|_U$ and $\F\O_{C^\dagger}|_{U^\dagger}$ are trivial, where 
$U^\dagger:=b^{-1}(U)$. For each 
such $U$, let $\mu$ and $h$ be generators of $\w_C|_U$ and 
$\F\O_{C^\dagger}|_{U^\dagger}$, respectively. Set $\mu^\dagger=h\mu$. Then 
\eqref{Fw} implies that $\mu^\dagger$ is a generator of 
$\w_{C^\dagger}|_{U^\dagger}$. Let $\partial$ (resp. $\partial^\dagger$) be the 
$k$-linear derivations of 
$\Gamma(U,O_C)$ (resp. $\Gamma(U^\dagger,O_{C^\dagger})$) such that 
$\gamma df=\partial f\mu$ for each regular 
function $f$ on $U$ (resp. 
$\gamma^\dagger d^\dagger f=\partial^\dagger f\mu^\dagger$ 
for each regular function $f$ on $U^\dagger$). 
It follows from the commutativity of Diagram \eqref{WCC'} 
that $\partial f=h\partial^\dagger f$ for each regular 
function $f$ on $U$.

Let $\sigma$ be a generator for $\L|_U$. Fix a basis $\beta:=(v_0,\dots,v_r)$ of 
$V$. Then there are regular functions $f_0,\dots,f_r$ on $U$ such that 
$\epsilon(v_i)|_U=f_i\sigma$ for each $i=0,\dots,r$. 
Then $W(\L,\epsilon)$ is defined on 
$U$ by the zero scheme of
$$
w:=\begin{vmatrix}
f_0 & \dots & f_r\\
(h\partial^{\dagger}) f_0 & \dots & (h\partial^{\dagger}) f_r\\
\vdots & \ddots &\vdots\\
(h\partial^{\dagger})^rf_0 & \dots & (h\partial^{\dagger})^rf_r\\
\end{vmatrix}.
$$
Using the multilinearity of the determinant and 
the product rule for derivations, we 
get that 
$$
w=h^{\binom{r+1}{2}}w',\quad\text{where}\quad
w':=\begin{vmatrix}
f_0 & \dots & f_r\\
\partial^{\dagger} f_0 & \dots & \partial^{\dagger} f_r\\
\vdots & \ddots &\vdots\\
(\partial^{\dagger})^rf_0 & \dots & (\partial^{\dagger})^rf_r\\
\end{vmatrix}.
$$
But $w'$ is exactly the regular function defining the Weierstrass divisor of the 
induced linear system $(b^*\L,\epsilon^\dagger)$ on $U^\dagger$. It follows that
$$
b^*W(\L,\epsilon)=W(b^*\L,\epsilon^\dagger)+\binom{r+1}{2}Z^\dagger.
$$
Now, since $W(\L,\epsilon)$ is a Cartier divisor of $C$, we have that 
$$
b_*[b^*W(\L,\epsilon)]=[W(\L,\epsilon)].
$$
Thus, using \eqref{FZZ}, we get
$$
[W(\L,\epsilon)]-b_*[W(b^*\L,\epsilon^\dagger)]=\binom{r+1}{2}b_*[Z^\dagger]
=(r+1)rR_b,
$$
proving \eqref{VLL}.

We will now tackle the general case, where $\I$ is not assumed invertible. 
Let $\varphi\:\I\hookrightarrow\L$ be an injection into an invertible 
sheaf $\L$. Then $\varphi$ induces an injection 
$\varphi^\dagger\:\I^b\hookrightarrow b^*\L$. 
Clearly, we have a commutative diagram 
of injections:
\begin{equation}\label{CDIL}
\begin{CD}
\I @>h^b_\I >> b_*\I^b\\
@V\varphi VV @Vb_*\varphi^\dagger VV\\
\L @>h^b_\L >> b_*b^*\L.
\end{CD}
\end{equation}

Let $Y$ be the closed subscheme of $C$ such that $\varphi(\I)=\I_{Y/C}\L$. Let 
$Y^\dagger:=b^{-1}(Y)$. Then 
$\varphi^\dagger(\I^b)=\I_{Y^\dagger/C^\dagger}b^*\L$. Now, 
$$
b_*[Y^\dagger]=b_*[\text{Coker}(\varphi^\dagger)]=
[\text{Coker}(b_*\varphi^\dagger)].
$$
Thus, from the commutativity of \eqref{CDIL},
\begin{equation}\label{RRI}
\begin{aligned}
b_*[Y^\dagger]-[Y]=&[\text{Coker}(b_*\varphi^\dagger)]-[\text{Coker}(\varphi)]\\
=&[\text{Coker}(h^b_\L)]-[\text{Coker}(h^b_\I)]\\
=&R_b-R_b(\I).
\end{aligned}
\end{equation}

From the definitions of the Weierstrass cycles we have:
$$
R(\I,\epsilon)-b_*R(\I^b,\epsilon^\dagger)=R(\L,\delta)-b_*R(b^*\L,\delta^\dagger)+
(r+1)(b_*[Y^\dagger]-[Y]),
$$
where $\delta:=\Gamma(\varphi)\epsilon$ and 
$\delta^\dagger:=\Gamma(\varphi^\dagger)\epsilon^\dagger$. 
Thus, using \eqref{VLL} and \eqref{RRI}, we get
\begin{align*}
R(\I,\epsilon)-b_*R(\I^b,\epsilon^\dagger)=&(r+1)rR_b+(r+1)(R_b-R_b(\I))\\
=&(r+1)^2R_b-(r+1)R_b(\I).
\end{align*}
\end{proof}

\begin{remark}\label{rmk1}\setcounter{equation}{0} Formula \eqref{VLL}, a special 
case of Proposition \ref{rambir}, was obtained in \cite{GL}, Prop.~1.6 , p.~4845.
\end{remark}

\section{The intrinsic Weierstrass scheme}\label{invramsch}

\begin{proposition}\label{conductor}\setcounter{equation}{0} 
Let $C$ be a Gorenstein curve. 
Let $b\:C^\dagger\to C$ be a birational map of curves and $\F$ its conductor 
ideal. Let $\I$ be a torsion-free, rank-$1$ sheaf on $C$ such that 
$Hom(\I,\O_C)^b$ is invertible. 
Then there are an invertible sheaf $\L$ on $C$ and an injection 
$\varphi\:\I\hookrightarrow \L$ such that 
$\varphi(\I)\supseteq\F\L$.
\end{proposition}

\begin{proof} Since $Hom(\I,\O_C)^b$ is invertible, for each $P\in C$ there is an 
element $s_P\in Hom(\I,\O_C)_P$ such that 
$$
s_P(b_*\O_{C^\dagger})_P=b_*(Hom(\I,\O_C)^b)_P.
$$
Let $\M$ be the subsheaf of $Hom(\I,\O_C)$ such that 
$\M_P=s_P\O_{C,P}$ for each $P\in C$. Then $\M$ is an invertible subsheaf 
of $Hom(\I,\O_C)$ satisfying $b^*\M=Hom(\I,\O_C)^b$. So, 
the natural injection $h^b_\M\:\M\hookrightarrow b_*b^*\M$ extends to an injection 
$Hom(\I,\O_C)\to b_*b^*\M$. (This argument appeared in the proof of 
\cite{abelpres}, Lemma 4.2, p.~5975.)

Since $C$ is Gorenstein,
$$
\I=Hom(Hom(\I,\O_C),\O_C).
$$
Thus, taking duals in the inclusion $\M\hookrightarrow Hom(\I,\O_C)$, 
we obtain an injection $\I\to Hom(\M,\O_C)$ whose image contains the 
image of $Hom(h^b_\M,\O_C)$. But, since $\M$ is invertible, $h^b_\M=b^\#\ox\M$, 
where $b^\#$ is the comorphism to $b$. Thus, the image of $Hom(h^b_\M,\O_C)$ is 
$$
\text{Im}(Hom(b^\#,\O_C))Hom(\M,\O_C).
$$
And the image of $Hom(b^\#,\O_C)$ is $\F$.
\end{proof}

\begin{subsct}\label{jets}\setcounter{equation}{0}
{\it Wronskians.} 
Let $C$ be a Gorenstein curve over an algebraically closed field $k$ of 
characteristic zero, and denote by $\w_C$ its 
sheaf of regular differentials. 
Let $\I$ be a torsion-free, rank-1 sheaf on $C$, and 
$\epsilon\:V\to\Gamma(C,\I)$ be a map of vector spaces. Assume 
$(\I,\epsilon)$ is strongly nondegenerate, and let $r$ denote its rank. 
 
Let $\varphi\:\I\hookrightarrow\L$ be an injection into an invertible 
sheaf $\L$. It induces a map 
$\varphi^{\ox n}\:\I^{\ox n}\to\L^{\ox n}$ for each integer $n>0$. 
Since $\L^{\ox n}$ is 
torsion-free, this map factors through an injection
 $\varphi^n\:\I^n\hookrightarrow\L^{\ox n}$, where $\I^n$ is $\I^{\ox n}$ modulo 
torsion.

Let $\epsilon'\:V\to\Gamma(C,\L)$ be the composition of $\epsilon$ with 
$\Gamma(\varphi)$. 
Let $w$ be a Wronskian of $(\L,\epsilon')$, a 
global section of $\L^{\ox r+1}\ox\w_C^{\ox\binom{r+1}{2}}$. 
We claim that $w$ factors through 
$$
\varphi^{r+1}\ox 1\:\I^{r+1}\ox\w_C^{\ox\binom{r+1}{2}}\lra \L^{\ox r+1}
\ox\w_C^{\ox\binom{r+1}{2}},
$$
yielding a global section of 
$\I^{r+1}\ox\w_C^{\ox\binom{r+1}{2}}$. We call this section a \emph{Wronskian} of 
$(\I,\epsilon)$.

To prove the claim, observe first that we may 
replace $\varphi$ by the composition $\lambda\varphi$, 
for any injection $\lambda\:\L\hookrightarrow\N$ into any invertible sheaf $\N$, 
since the composition of $w$ with the induced map, 
$$
\lambda^{\ox r+1}\ox 1\:\L^{\ox r+1}\ox\w_C^{\ox\binom{r+1}{2}}\lra 
\N^{\ox r+1}\ox\w_C^{\ox\binom{r+1}{2}},
$$
is a Wronskian of the induced $(\N,\epsilon'')$, where 
$\epsilon''$ is the composition of $\epsilon'$ with $\Gamma(\lambda)$. 
In particular, it follows from Lemma \ref{inj2} that 
the claimed factorization of $w$ occurs in general if it occurs for a 
particular injection $\varphi$. Also, it follows that the 
induced global section of $\I^{r+1}\ox\w_C^{\ox\binom{r+1}{2}}$ 
does not depend on the choice of $\varphi$; in other words, a Wronskian is 
well-defined modulo multiplication by $k^*$.

Thus, by Proposition \ref{conductor}, 
we may assume that $\varphi(\I)\supseteq\F\L$, where $\F$ is 
the conductor ideal of the normalization map $b\:C^\dagger\to C$. 
We proceed now as in the proof of Proposition \ref{rambir}. 
More precisely, cover $C$ by affine open 
subschemes $U$ such that $\w_C|_U$, $\L|_U$ and $\F\O_{C^\dagger}|_{U^\dagger}$ 
are trivial, where 
$U^\dagger:=b^{-1}(U)$. Let $\mu$ and $h$ be generators of $\w_C|_U$ and 
$\F\O_{C^\dagger}|_{U^\dagger}$, respectively. 
Let $\partial$ be the $k$-linear derivation of 
$\Gamma(U,\O_C)$ such that $\gamma df=\partial f\mu$ for each 
regular function $f$ on $U$. Then $\partial f=h\partial^\dagger f$ for each regular 
function $f$ on $U$, where $\partial^\dagger$ is a $k$-linear derivation of 
$\Gamma(U^\dagger,O_{C^\dagger})$.

Let $\sigma$ be a generator for $\L|_U$. 
Let $I\subseteq\Gamma(U,\O_C)$ be the ideal such that 
$\Gamma(U,\varphi(\I))=I\sigma$. Fix a basis $\beta:=(v_0,\dots,v_r)$ of 
$V$. Then there are $f_0,\dots,f_r\in I$ such that 
$\varphi\epsilon(s_i)|_U=f_i\sigma$ for each $i=0,\dots,r$. 

With the trivializations taken above, a wronskian of $(\L,\Gamma(\varphi)\epsilon)$ 
can be identified with
$$
u:=\begin{vmatrix}
f_0 & \dots & f_r\\
(h\partial^\dagger) f_0 & \dots & (h\partial^\dagger) f_r\\
\vdots & \ddots &\vdots\\
(h\partial^\dagger)^rf_0 & \dots & (h\partial^\dagger)^rf_r\\
\end{vmatrix}.
$$
We need to show that $u\in I^{r+1}$. 
Now, if $f$ is a regular function on $U$, then, by composition,  $f$ is also 
regular on $U^\dagger$; hence $h\partial^\dagger f\in\Gamma(U,\F)$, 
and in particular, 
$h\partial^\dagger f\in I$. 
Thus, all the entries of the matrix above are elements of $I$, and hence 
$w\in I^{r+1}$.
\end{subsct}

\begin{subsct}\label{zeroscheme}\setcounter{equation}{0}
{\it Zero schemes.} Let $C$ be a curve over an algebraically closed 
field $k$. Let $\I$ be a coherent sheaf on $C$ and $s$ a 
global section of $\I$.

We may view $s$ as a map $\sigma\:\O_C\to\I$. Taking duals, we obtain a map 
$\sigma^*\: Hom(\I,\O_C)\to\O_C$ 
whose image is the sheaf of ideals of a closed subscheme of $C$, 
which we denote by $Z_s$. If $s$ is generically nonzero, 
then the dual map has finite cokernel, that is, $Z_s$ is finite. 
We call $Z_s$ the \emph{zero scheme} of $s$. If $\I$ is invertible, 
then $Z_s$ is the usual zero scheme.

Assume $C$ is Gorenstein and $\I$ is torsion-free, rank-1. 
We claim that
$$
\deg[Z_s]=\deg\I.
$$
Indeed, $[Z_s]=[\text{Coker}(\sigma^*)]$. Let 
$\w_C$ be the sheaf of regular differentials 
of $C$. Since 
$$
Ext^1(\I,\O_C)=Ext^1(\I,\w_C)\ox\w_C^{-1}=0,
$$
where the first equality holds because $\w_C$ is invertible and the second by 
\eqref{Ext10}, we have that 
$$
\text{Coker}(\sigma^*)=Ext^1(\text{Coker}(\sigma),\O_C).
$$
Now, for any coherent sheaf $\G$ on $C$ with finite support,
$$
Hom(\G,\O_C)=Ext^2(\G,\O_C)=0
$$
by local duality. Since any such $\G$ can be viewed as an extension of a 
skyscraper sheaf of length 1 by a sheaf of smaller length, and since 
$Ext^1(k,\O_{C,P})\cong k$ for each $P\in C$, because $C$ is Gorenstein, 
it follows that
$$
[Ext^1(\G,\O_C)]=[\G].
$$ 
Then 
\begin{equation}\label{CCoker}
[\text{Coker}(\sigma^*)]=[\text{Coker}(\sigma)].
\end{equation}
Finally, $\deg[\text{Coker}(\sigma)]=\deg\I$, proving the claim.
\end{subsct}

\begin{subsct}\label{jetsintr}\setcounter{equation}{0}
{\it The intrinsic Weierstrass scheme.} Keep the setup of 
Subsection~\ref{jets}. Since $(\I,\epsilon)$ is strongly nondegenerate, 
a Wronskian is generically nonzero, and hence its zero scheme is a 
finite subscheme of $C$ whose associated cycle has degree  
$$
\deg\I^{r+1}+(r+1)r(g-n),
$$
where $\I^{r+1}$ is $\I^{\ox r+1}$ modulo torsion, and 
$n$ is the number of connected components of $C$. 
We denote this subscheme by $Z(\I,\epsilon)$ and call it 
the \emph{intrinsic Weierstrass scheme} of $(\I,\epsilon)$.
\end{subsct}

\begin{subsct}\label{defect}\setcounter{equation}{0}
{\it The $n$-th defect.} Let $C$ be a curve and $\I$ a 
torsion-free, rank-1 sheaf on $C$. Let 
$\varphi\:\I\hookrightarrow\L$ be an injection into an invertible sheaf $\L$. 
Let $Y\subset C$ 
be the subscheme such that $\varphi(\I)=\I_{Y/C}\L$. 
For each integer $n>0$, let 
$\varphi^n\:\I^n\to\L^{\ox n}$ be the induced injection, 
where $\I^n$ is $\I^{\ox n}$ 
modulo torsion. Then $\text{Im}(\varphi^n)=\I_{Y/C}^n\L^{\ox n}$. 
We let $Y^n\subset C$ 
be the closed subscheme defined by $\I_{Y/C}^n$, or equivalently, 
by the property that
$\varphi^n(\I^n)=\I_{Y^n/C}\L^{\ox n}$. Set
$$
\Delta^n(\I):=[Y^n]-n[Y].
$$
We call $\Delta^n(\I)$ the \emph{$n$-th defect} of $\I$. Of course, 
the $Y^n$ are divisors and $Y^n=nY$ 
where $\I$ is invertible. Thus the $n$-th defect 
is supported on the singular locus of $C$. 

Given an injection $\lambda\:\L\hookrightarrow\N$ into an invertible sheaf $\N$, 
let $Z$ be the subscheme 
of $C$ such that $\lambda(\L)=\I_{Z/C}\N$. Since $\L$ and $\N$ are invertible, 
$Z$ is an effective divisor. Thus 
$$
(\lambda\varphi)^n(\I)=\lambda^{\ox n}\varphi^n(\I)=\I_{Y^n/C}\I_{Z/C}^n\N^{\ox n}.
$$
The subscheme given by $\I_{Y^n/C}\I_{Z/C}^n$ has associated cycle $[Y^n]+n[Z]$. 
Thus, using Lemma \ref{inj2}, it follows that $\Delta^n(\I)$ does not depend 
on the choice of injection $\varphi$. 

\end{subsct}

\begin{proposition} Let $C$ be a Gorenstein curve 
over an algebraically closed field 
of characteristic $0$ and $(\I,\epsilon)$ a strongly 
nondegenerate (generalized) linear system 
of rank $r$ on $C$. Then
$$
R(\I,\epsilon)=[Z(\I,\epsilon)]+\Delta^{r+1}(\I).
$$
\end{proposition}

\begin{proof} Let $\varphi\:\I\hookrightarrow\L$ be 
an injection into an invertible sheaf $\L$. Let $Y\subset C$ 
be the subscheme such that $\varphi(\I)=\I_{Y/C}\L$. 
Let $Y^{r+1}\subset C$ be defined 
by $\I^{r+1}_{Y/C}$. Then  
$$
[Z(\I,\epsilon)]=[W(\L,\epsilon)]-\Big[\frac{Hom(\I_{Y^{r+1}/C},\O_C)}{\O_C}\Big]
=[W(\L,\epsilon)]-[Y^{r+1}],
$$
where the last equality follows from \eqref{CCoker}. Since 
$$
R(\I,\epsilon)=[W(\L,\epsilon)]-(r+1)[Y],
$$
the statement follows.
\end{proof}

\section{Families}\label{degtorfree}

\begin{subsct}\label{famcur}\setcounter{equation}{0}
{\it Families of curves and sheaves.} Let $\pi\:C\to S$ be a projective, 
flat map whose geometric fibers 
are curves. We call $\pi$ a \emph{family of curves}. 

Let $\I$ be a coherent sheaf on $C$ which is flat over $S$, 
and restricts to a torsion-free, 
rank-1 sheaf on every geometric fiber of $\pi$. We call $\I$ a 
\emph{family of torsion-free, rank-1 sheaves} along $\pi$. 
Of course, an invertible sheaf on 
$C$ is a family of torsion-free, rank-1 sheaves along $\pi$. 
We say that $\I$ has \emph{relative degree} $d$ 
if the restriction of $\I$ to each geometric fiber of $\pi$ has degree $d$. 
By flatness, if $S$ is connected, then $\I$ has a relative degree. 
For each geometric point $s$ of $S$, set $\I(s):=\I|_{C(s)}$.
\end{subsct}

\begin{proposition}\label{ILs} 
Let $\pi\:C\to S$ be a family of Gorenstein 
curves, $\I$ a family of torsion-free, rank-1 sheaves along $\pi$ and $\L$ an 
invertible sheaf on $C$. Then $Hom(\I,\L)$ is a family of torsion-free, rank-1 
sheaves along $\pi$. Furthermore, 
if $s$ is a geometric point of $S$ then the natural map 
$$
Hom(\I,\L)(s)\lra Hom(\I(s),\L(s))
$$
is an isomorphism. Also, if $\I$ has relative degree $d$ and $\L$ has relative 
degree $e$, then $Hom(\I,\L)$ has relative degree $e-d$.
\end{proposition}

\begin{proof} As pointed out in Subsection \ref{tfrk1sh}, namely \eqref{Ext10},
$$
Ext^1(\I(s),\w_{C(s)})=0
$$
for every geometric point $s$ of $S$, where $\w_{C(s)}$ is the 
sheaf of regular differentails of $C(s)$. Since $C(s)$ is Gorenstein, 
$\w_{C(s)}$ is invertible, and thus
$$
Ext^1(\I(s),\L(s))=Ext^1(\I,\w_{C(s)})\ox\w_{C(s)}^{-1}\ox\L(s)=0.
$$
Hence $Ext^1(\I,\L)=0$ by \cite{AK}, Thm.~1.10, p.~61.
It follows now from \cite{AK}, Thm.~1.9, p.~59, that 
$Hom(\I,\L)$ is flat over $S$ and the natural map
\begin{equation}\label{Homs}
Hom(\I,\L)(s)\lra Hom(\I(s),\L(s))
\end{equation}
is an isomorphism for every geometric point $s$ of $S$. Thus $Hom(\I,\L)$ 
is a family of torsion-free, rank-1 sheaves along $\pi$. Finally, since
$$
Hom(\I(s),\L(s))=Hom(\I(s),\O_{C(s)})\ox\L(s),
$$
the last statement of the proposition follows from the isomorphism \eqref{Homs} and 
Equations \eqref{degdeg} and \eqref{deg-deg}.
\end{proof}

\begin{subsct}\label{zeroschemeII}\setcounter{equation}{0}
{\it Zero schemes, II.} Let $\pi\:C\to S$ be a family of 
Gorenstein curves. Let $\I$ be a family of torsion-free, rank-1 sheaves of relative 
degree $d$ along $f$, and $s$ a global section of $\I$. 

As in Subsection \ref{zeroscheme}, the section $s$ corresponds to a map 
$\sigma\:\O_C\to\I$. The image of its dual, $\sigma^*\: Hom(\I,\O_C)\to\O_C$, 
is a sheaf of ideals of a closed subscheme of $C$, which we denote by $Z_s$. It 
follows from Proposition \ref{ILs} that the formation of $Z_s$ commutes 
with base change. So, if $s$ is nonzero at every generic point of every geometric 
fiber of $\pi$, then $\sigma^*$ is injective with $S$-flat cokernel; in other 
words, $Z_s$ is flat over $S$ of relative length $d$.
\end{subsct}

\begin{lemma}\label{lemIL} Let $\pi\:C\to S$ be a family of Gorenstein 
curves and $\I$ a family of torsion-free, 
rank-1 sheaves along $\pi$. 
Then, for each $s\in S$ whose residue field $\kappa(s)$ is infinite, 
there are an open neighborhood $U$ of $s$ in $S$ 
and an injection $\I|_{\pi^{-1}(U)}\hookrightarrow\L$ with $S$-flat cokernel, 
where $\L$ is an invertible sheaf on $\pi^{-1}(U)$. 
\end{lemma}

\begin{proof} Let $\O_C(1)$ be a relatively ample sheaf on $C$ over $S$. 
For each integer $m$ sufficiently large, $Hom(\I(s),\O_C(m)(s))$ is generated 
by global sections and the base-change map 
\begin{equation}\label{Homs2}
\text{Hom}(\I,\O_C(m))(s) \lra \Gamma(C(s),Hom(\I,\O_C(m))(s))
\end{equation}
is surjective. Pick such $m$. Since $\kappa(s)$ is infinite, there 
is a section of $Hom(\I(s),\O_C(m)(s))$ which is nonzero at every generic point of 
$C(s)$. This section corresponds to an injection 
$\I(s)\hookrightarrow\O_C(m)(s)$. Since \eqref{Homs} is an isomorphism for 
$\L:=\O_C(m)$, and 
\eqref{Homs2} is surjective, up to replacing $S$ by an open neighborhood of $s$, we 
may assume that the injection lifts to a map $\varphi\:\I\to\O_C(m)$. Since 
$\varphi(s)$ is injective, up to replacing $S$ by an open neighborhood of $s$, it 
follows that $\varphi$ is injective with $S$-flat cokernel.
\end{proof}

\begin{proposition}\label{proplim} 
Let $S$ be the spectrum of a discrete valuation ring with infinite 
residue field. 
Let $\pi\:C\to S$ be a family of Gorenstein curves and $\I$ a family 
of torsion-free, rank-1 sheaves along $\pi$. 
If the restriction of $\I$ to the generic 
fiber of $\pi$ is invertible, then the $n$-th defect of the 
restriction of $\I$ to each 
geometric fiber of $\pi$ is nonnegative, for each integer $n>0$.
\end{proposition}

\begin{proof} Let $u$ be the generic point 
and $s$ the special point of $S$. Of course, since $\I$ is invertible on $C(u)$, 
the restriction of $\I$ to any field extension of $C(u)$ has defect 0. 

By Lemma \ref{lemIL}, there is an injection $\varphi\:\I\to\L$ with 
$S$-flat cokernel into an invertible sheaf $\L$. Let $Y\subset C$ be the 
closed subscheme such that $\varphi(\I)=\I_{Y/C}\L$. Since $\L$ is invertible, 
$Y$ is flat over $S$.

For each integer $n>0$ consider the induced 
map $\varphi^{\ox n}\:\I^{\ox n}\to\L^{\ox n}$. 
Since $S$ is the spectrum of a discrete valuation 
ring, there is a surjection $\rho\:\L^{\ox n}\to\H$ onto an $S$-flat 
coherent sheaf $\H$ such that 
\begin{equation}\label{defcplx}
\begin{CD}
0 @>>> \I^{\ox n} @>\varphi^{\ox n}>> \L^{\ox n} @>\rho >> \H @>>> 0
\end{CD}
\end{equation}
is a complex which is exact on $C(u)$. Let $Z\subset C$ be the 
closed subscheme such that $\I_{Z/C}\L^{\ox n}=\text{Ker}(\rho)$. Since 
$\varphi(s)$ is generically bijective, $Z$ is finite over $S$. Also, since 
$\H$ is $S$-flat, so is $Z$.

Notice that, since \eqref{defcplx} is exact on $C(u)$ and $\I(u)$ is invertible, 
$Y(u)$ and $Z(u)$ are divisors satisfying $Z(u)=nY(u)$. It follows from 
\cite{E3}, Prop.~3.4, that $[Z(t)]=n[Y(t)]$, for any geometric point $t$ of $S$ 
above $s$. 
Furthermore, since \eqref{defcplx} is a complex, $Z(t)\subseteq Y^n_t$, where 
$Y^n_t$ is the closed subscheme of $C(t)$ satisfying
$$
\varphi(t)^{\ox n}(\I(t)^{\ox n})=\I_{Y^n_t/C(t)}\L(t)^{\ox n}.
$$
(Equivalently, the sheaf of ideals of $Y^n_t$ is $\I^n_{Y(t)/C(t)}$.) Thus
$$
\Delta^n(\I(t))=[Y^n_t]-n[Y(t)]\geq [Z(t)]-n[Y(t)]=0.
$$
\end{proof}

\section{Degenerations of linear systems}\label{deglinsys}

\begin{subsct}\label{famlinsys}\setcounter{equation}{0}
{\it Families of linear systems.} Let $\pi\:C\to S$ be a family of 
curves. Let $\I$ be a family of torsion-free, rank-1 sheaves along $\pi$ 
of relative 
degree $d$. 
Let $\V$ be a locally free sheaf of constant rank $r+1$ on $S$, 
for a certain integer $r$, and $\epsilon\:\V\to \pi_*\I$ a map. We 
say that $(\I,\epsilon)$ is a \emph{family of} (generalized) 
\emph{linear systems} of degree 
$d$ and rank $r$ along $\pi$. The terminology is justified because 
for each geometric point $s$ 
of $S$ the composition
$$
\begin{CD}
\epsilon_s\:\V(s) @>\epsilon(s)>> \pi_*\I(s) @>>> \Gamma(C(s),\I(s)) 
\end{CD}
$$
gives rise to a (generalized) linear system of degree 
$d$ and rank $r$ on the fiber $C(s)$, 
where the second map is the base-change map. We call $(\I(s),\epsilon_s)$ the 
\emph{induced} (generalized) \emph{linear system} on $C(s)$. We say that 
$(\I,\epsilon)$ is 
(\emph{strongly}) \emph{nondegenerate} if the induced linear system on 
every geometric fiber of $\pi$ is 
(strongly) nondegenerate.
\end{subsct}

\begin{subsct}\label{relfundclass}\setcounter{equation}{0}
{\it The relative fundamental class.} Let $\pi\:C\to S$ be a family of curves. 
Let $\gamma\:\Omega^1_{C/S}\to\w_{C/S}$ be a map to a coherent 
$S$-flat sheaf $\w_{C/S}$ whose restriction to every geometric fiber is 
the fundamental class of that fiber. More precisely, $\gamma$ is 
assumed such that, for every geometric point $s$ of $S$, there is a 
commutative diagram of maps
$$
\begin{CD}
\Omega^1_{C/S}|_{C(s)} @>\gamma|_{C(s)}>> \w_{C/S}|_{C(s)}\\
@VVV @VVV\\
\Omega^1_{C(s)/\kappa(s)} @>>> \w_{C(s)},
\end{CD}
$$
where the vertical maps are isomorphisms, the left one canonical, and the bottom 
map is the fundamental class of $C(s)$. We call such a map a 
\emph{relative fundamental class} of $C/S$.

According to \cite{AE}, Thm.~III.1, p.~81, there is a map 
$\gamma\:\Omega^1_{C/S}\to\w_{C/S}$, the unique one satisfying a certain 
trace property. Actually, the target of the map in loc.~cit.~is a 
complex, the relative dualizing complex. 
However, in our case the dualizing complex can be 
replaced by a single sheaf; see \cite{Conrad}, Thm.~3.5.1, p.~155. 
The trace property, \cite{AE}, Def.~II.3, p.~79, 
asserts a certain compatibility between $\gamma$ and the trace map 
$h_*h^*\Omega^1_{X/S}\to\Omega^1_{X/S}$ for the sheaf of regular differentials 
$\Omega^1_{X/S}$ of 
a smooth family of curves $X/S$, if there is a finite and flat $S$-map 
$h\: C\to X$. The analogous compatibility condition 
holds when one restricts to a 
geometric fiber, by \cite{AE}, Propri\'et\'e 1, p.~78. 
This compatibility condition is also satisfied by the fundamental class of each 
geometric fiber; at least, this is shown for irreducible curves in 
\cite{Kunz}, Satz 5.6, p.~105.  Finally, \cite{AE}, Prop.~II.3.1, p.~79, 
asserts that the trace property is satisfied by a unique map, thus showing that 
$\gamma$ restricts to the fundamental class on each geometric fiber. 

Alternatively, it is shown in \cite{KW}, Thm.~5.26, p.~112, 
that a map $\gamma\:\Omega^1_{C/S}\to\w_{C/S}$ exists 
if $S$ has no embedded points and can be covered by affine open subschemes whose 
rings of functions are Noetherian, universally Japanese and 
universally catenary. There is not a clear statement in \cite{KW} to the 
effect that $\gamma$ restricts to the fundamental class on each geometric fiber, 
though Cor.~5.29, p.~116 and Prop.~4.36, p.~89 should imply this, at least if 
the geometric fibers of $\pi$ are irreducible. 

Also, if the geometric fibers of $C/S$ are locally complete intersections, 
then the existence of a map $\gamma\:\Omega^1_{C/S}\to\w_{C/S}$ was pointed out 
in \cite{E1}. Again, it is not shown in \cite{E1} that $\gamma$ restricts to the 
fundamental class on each geometric fiber. This can be derived from the 
fact that the construction given to $\gamma$ commutes with base change, and from 
\cite{Li}, Cor.~13.7, p.~114, though the latter is stated only for irreducible 
curves.

Finally, a map $\gamma\:\Omega^1_{C/S}\to\w_{C/S}$ is constructed in \cite{KK} 
without hypotheses on the basis $S$ or on the fibers of $\pi$. Again, it is 
not clearly stated whether this $\gamma$ restricts to the fundamental class on 
every geometric fiber of $f$, but \cite{KK} is a work in progress.

Since $\w_{C/S}$ restricts to the sheaf of regular differentials on each geometric 
fiber, it follows that $\w_{C/S}$ is a 
family of torsion-free, rank-1 sheaves along $\pi$. Furthermore, if the 
geometric fibers of $\pi$ are Gorenstein, then $\w_{C/S}$ is 
invertible. 
\end{subsct}

\begin{subsct}\label{famcurdiv}\setcounter{equation}{0}
{\it Weierstrass schemes.} Let $\pi\:C\to S$ be a family of 
Gorenstein curves. Let $\L$ be an invertible sheaf on $C$, and 
$\epsilon\:\V\to f_*\L$ a map from a locally free sheaf $\V$ of 
constant rank $r+1$, for a certain integer $r$. 
Let $\gamma\:\Omega^1_{C/S}\to\w_{C/S}$ be a relative fundamental class. 

We can associate to $(\L,\epsilon)$ a global section $w(\L,\epsilon)$ of 
\begin{equation}\label{LWV}
\L^{\ox r+1}\ox\w_{C/S}^{\ox\binom{r+1}{2}}\ox\pi^*\bigwedge^{r+1}\V^*,
\end{equation}
which we will call the \emph{Wronskian} of $(\L,\epsilon)$. Its 
zero scheme will be called the \emph{Weierstrass scheme} of $(\L,\epsilon)$.

The Wronksian is constructed in a way similar to that of the Weierstrass divisor 
of Subsection \ref{ramdiv}, as follows. Cover $S$ by  
open subschemes $U$ such that $\V|_U$ is trivial, and for each such $U$ 
cover $\pi^{-1}(U)$ by open subschemes $U'$ such that 
$\w_{C/S}|_{U'}$ and $\L|_{U'}$ 
are trivial. For each $U$, let $\beta=(v_0,\dots,v_r)$ be a basis of the free 
$\Gamma(U,\O_S)$-module $\Gamma(U,\V)$. And for each $U'$ let 
$\mu\in\Gamma(U',\w_{C/S})$ generating $\w_{C/S}|_{U'}$ and 
$\sigma\in\Gamma(U',\L)$ generating $\L|_{U'}$. Then there is a 
$\Gamma(U,\O_S)$-linear derivation $\partial$ of $\Gamma(U',\O_C)$ such that 
$\gamma df=\partial f\mu$ for each regular funtion $f$ on $U'$. And there are 
regular functions $f_0,\dots,f_n$ on $U'$ such that 
$\epsilon(v_i)|_{U'}=f_i\sigma$ for each $i=0,\dots,n$.

Form the wronskian determinant:
$$
w(\beta,\sigma,\mu):=
\begin{vmatrix}
f_0 & \dots & f_r\\
\partial f_0 & \dots & \partial f_r\\
\vdots & \ddots &\vdots\\
\partial^rf_0 & \dots & \partial^rf_r\\
\end{vmatrix}.
$$
As in Subsection \ref{ramdiv}, the multilinearity of the determinant and the 
product rule of derivations imply that the $w(\beta,\sigma,\mu)$ patch to a 
global section $w(\L,\epsilon)$ of the sheaf \eqref{LWV}.

It is also clear from the above description that the Wronskian of $w(\L,\epsilon)$ 
restricts to a Wronskian of $(\L(s),\epsilon_s)$ 
for each geometric point $s$ of $S$, 
once isomorphisms $\bigwedge^{r+1}\V(s)\risom\kappa(s)$ and 
$\w_{C/S}(s)\risom\w_{C(s)}$ are chosen, the latter such that its composition 
with $\gamma(s)$ is the fundamental class of $C(s)$. It follows that 
the Weierstrass scheme of $(\L,\epsilon)$ is a relative Cartier divisor over $S$, 
restricting to the Weierstrass divisor of $(\L(s),\epsilon_s)$ for each 
geometric point $s$ of $S$, if $(\L,\epsilon)$ 
is strongly nondegenerate.

Furthermore, functoriality holds for $w(\L,\epsilon)$: If $\varphi\:\L\to\M$ 
is a map to an invertible sheaf $\M$, and 
$\epsilon':=(\pi_*\varphi)\epsilon$, then 
$$
w(\M,\epsilon')=(\varphi^{\ox r+1}\ox 1)w(\L,\epsilon),
$$
where $1$ is the identity map of
$\w_{C/S}^{\ox\binom{r+1}{2}}\ox\pi^*\bigwedge^{r+1}\V^*$. 
\end{subsct}

\begin{theorem}\label{thm} 
Let $S$ be the spectrum of a discrete valuation ring with 
algebraically closed residue field of characteristic zero. Let $s$ and $u$ be 
its special and generic points, respectively. Let $\pi\:C\to S$ be a family of 
Gorenstein curves and $(\I,\epsilon)$ a family of strongly nondegenerate 
(generalized) linear systems along $\pi$. Assume that $\I(u)$ is invertible, and 
let $W\subset C$ be the schematic closure of the Weierstrass scheme of the 
linear system $(\I(u),\epsilon_u)$ induced on $C(u)$. Then $W(s)$ 
is a finite subscheme containing the intrinsic Weierstrass scheme 
$Z(\I(s),\epsilon_s)$ and 
whose associated cycle is $R(\I(s),\epsilon_s)$.
\end{theorem}

\begin{proof} By Lemma \ref{lemIL} there is an injection 
$\varphi\:\I\hookrightarrow\L$ into an invertible sheaf $\L$ whose cokernel is 
$S$-flat. Let $Y\subset C$ such that 
$\varphi(\I)=\I_{Y/C}\L$. Then $Y$ is $S$-flat. 
As in the proof of Proposition~\ref{proplim}, consider the induced 
map $\varphi^{\ox r+1}\:\I^{\ox r+1}\to\L^{\ox r+1}$ and 
the $S$-flat closed subscheme 
$Z\subset C$ such that 
$\text{Im}(\varphi^{\ox r+1})\subseteq\I_{Z/C}\L^{\ox r+1}$, with equality holding 
over $u$. Then $Z(s)\subseteq Y^{r+1}_s$, where 
$\I_{Y^{r+1}_s/C(s)}=\I^{r+1}_{Y(s)/C(s)}$. Also, since $\I(u)$ is invertible, 
$Y(u)$ is Cartier and $Z(u)=(r+1)Y(u)$.

Fix a relative fundamental class $\eta\:\Omega^1_{C/S}\to\w_{C/S}$. 
Set $\epsilon\:\V\to\pi_*\I$, where $\V$ is a locally free 
sheaf of constant rank, say $r+1$. Consider the induced family of linear systems 
$(\L,\epsilon')$, where we set $\epsilon':=\pi_*\varphi\circ\epsilon$, and its 
associated Wronskian $w(\L,\epsilon')$. Since the restriction 
$w(\L,\epsilon')(u)$ 
factors through $w(\I(u),\epsilon_u)$, it follows that 
$w(\L,\epsilon')$ factors through a global section $w$ of 
$$
\I_{Z/C}\L^{\ox r+1}\ox\w_{C/S}^{\ox\binom{r+1}{2}}.
$$
Since $Z$ is flat over $S$, the above sheaf is a family of torsion-free, rank-1 
sheaves along $\pi$. So, since $w$ restricts to the Wronskian 
$w(\I(u),\epsilon_u)$, under the identification $\I\risom\I_{Y/C}\L$ given by 
$\varphi$, it follows that the zero scheme of $w$ is $W$. Thus
\begin{align*}
[W(s)]&=[W(\L(s),\epsilon'_s)]-[Z(s)]\\
&=[W(\L(s),\epsilon'_s)]-(r+1)[Y(s)]=R(\I(s),\epsilon_s),
\end{align*}
where the second equality follows from \cite{E3}, Prop.~3.4, using that 
$Z(u)=(r+1)Y(u)$.

Finally, as seen in Subsection \ref{jets}, 
the Wronskian $w(\L(s),\epsilon'_s)$ factors 
through a section $w'$ of 
$$
\I^{r+1}_{Y(s)/C(s)}\L(s)^{\ox r+1}\ox\w_{C(s)}^{\ox\binom{r+1}{2}},
$$
which is contained in 
$$
\I_{Z(s)/C(s)}\L(s)^{\ox r+1}\ox\w_{C(s)}^{\ox\binom{r+1}{2}}
$$
because $Z(s)\subseteq Y^{r+1}_s$. So $w(s)$ factors through $w'$, and thus the 
zero scheme of $w'$ is contained in that of $w(s)$. In other words, we have 
$W(s)\supseteq Z(\I(s),\epsilon_s)$ as claimed.
\end{proof}

\begin{remark}\label{rmk} In the proof of Theorem \ref{thm}, the 
intrinsic Weierstrass scheme is the zero scheme of a section $w'$ of 
$$
\I(s)^{r+1}\ox\w_{C(s)}^{\ox\binom{r+1}{2}}.
$$
As seen in the proof, $W(s)$ is the zero scheme of a section $w$ obtained 
by composing $w'$ with 
$$
\psi\ox 1\:\I(s)^{r+1}\ox\w_{C(s)}^{\ox\binom{r+1}{2}}\lra
\J\ox\w_{C(s)}^{\ox\binom{r+1}{2}},
$$
where $\psi\:\I(s)^{r+1}\to \J$ is an injection into a torsion-free, rank-1 sheaf 
$\J$ satisfying $\deg\J=(r+1)\deg\I(s)$. From the proof, 
$$
\J=\I_{Z(s)/C(s)}\L(s)^{\ox r+1}.
$$ 
It can be shown, using a relative version of Lemma \ref{inj2} that $\J$ and 
$\psi$ do not depend on the choice of $\varphi$, but only on the family $\I$. 
So $W(s)$ depends only on $\I$. As Example \ref{example} below shows, $W(s)$ 
does not depend only on $\I(s)$ and $\epsilon_s$. 
\end{remark}

\begin{example}\label{example} Let $C$ be an irreducible plane curve of degree $e$ 
defined over an algebraically closed field $k$ of characteristic zero. 
Let $d$ be a positive integer smaller than $e$. 
For each $P\in C$, let $V_P$ be the vector subspace of 
$\Gamma(C,\O_C(d))$ generated by the degree-$d$ plane curves passing through 
$P$. Then 
$$
V_P\subseteq\Gamma(C,\I_{P/C}(d)).
$$
Let $(\I_{P/C}(d),\epsilon_P)$ be the corresponding (generalized) linear 
system, where $\epsilon_P\:V_P\to \Gamma(C,\I_{P/C}(d))$ denotes 
the inclusion. This 
system is (strongly) nondegenerate because $d<e$. If $P$ is on the 
nonsingular locus of $C$, denote by $W(C,P)\subseteq C$ the Weierstrass scheme of 
$(\I_{P/C}(d),\epsilon_P)$. 

For a very simple example, let $C$ be the nodal cubic, given by the equation 
$f(x,y,z)=0$, where 
$$
f(x,y,z):=y^2z-x^2z-x^3.
$$
Assume $d=1$. Let $Q$ be the node of $C$, given by $x=y=0$. Then 
$V_Q$ is generated by $x$ and $y$. 

Consider the inclusion $\varphi\:\I_{Q/C}(1)\hookrightarrow\O_C(1)$. 
The Weierstrass divisor of $(\O_C(1),\epsilon'_Q)$, where 
$\epsilon'_Q:=\Gamma(\varphi)\epsilon_Q$, is supported at $Q$ and given there 
by the equation
$$
\begin{vmatrix}
x & y\\
\frac{\partial f}{\partial y} & -\frac{\partial f}{\partial x}
\end{vmatrix}=0,
$$
that is, by the equation $x^3=0$. The associated cycle is thus $6Q$. So 
$$
R(\I_{Q/C}(1),\epsilon_Q)=4Q.
$$

Let $u:=x/z$ and $v:=y/z$. Then $u$ and $v$ generate the maximal ideal of 
$\O_{C,Q}$. The intrinsic Weierstrass scheme of $(\I_{Q/C}(1),\epsilon_Q)$ is 
supported at $Q$ and given there by the transporter ideal $(u^3:(u,v)^2)$, 
which is equal to $(u,v)^2$. So 
$$
[Z(\I_{Q/C}(1),\epsilon_Q)]=3Q.
$$

So, if $(\I_{Q/C}(1),\epsilon_Q)$ is a limit of ``true'' linear systems, the 
limit of the corresponding Weierstrass divisors is the subscheme $W(a,b)$ of $C$ 
supported at $Q$ and whose ideal at $Q$ is of the form
$$
(auv+bv^2)\O_{C,Q}+(u,v)^3
$$
for certain $a,b\in k$, with $a$ or $b$ nonzero. 

Furthermore, all $a$ and $b$ are possible. Indeed, consider first the 
family of pointed curves $(C_t,P_t)$, where $P_t$ is 
given by $x=0$ and $y=tz$, and $C_t$ is given by
$$
y^2z-x^3-x^2z-t^2z^3=0.
$$
For $t$ close to zero, but nonzero, $P_t$ is 
a nonsingular point of $C_t$. The limit of $W(C_t,P_t)$ can be computed to 
be $W(1,0)$. 

Finally, consider another family of pointed curves $(C_t,P_t)$, where $P_t$ 
is given by $x=tz$ and $y=2ctz$, for a fixed element $c\in k$, and 
where $C_t$ is given by
$$
y^2z+(t-1)x^2z-x^3-2ctyz^2+t^2z^3=0.
$$
Again, for $t$ close to zero, but nonzero, $P_t$ is 
a nonsingular point of $C_t$. The limit of $W(C_t,P_t)$ can be computed to 
be $W(c,1)$. 

The computations were done using \cocoa \cite{cocoa}.
\end{example}

\end{document}